\documentclass{article}
\usepackage[utf8]{inputenc}
\usepackage{amsmath}
\usepackage{amsthm}
\usepackage{amssymb}
\usepackage{array}
\usepackage{enumitem}
\usepackage{tikz}
\usepackage{array}
\usepackage{biblatex}
\usetikzlibrary{shapes}
\usetikzlibrary{arrows}
\usetikzlibrary{matrix}
\theoremstyle{definition}\newtheorem{prop}{Proposition}
\theoremstyle{definition}\newtheorem{conj}{Conjecture}
\theoremstyle{definition}\newtheorem{definition}[prop]{Definition}
\theoremstyle{definition}\newtheorem{lemma}[prop]{Lemma}
\theoremstyle{definition}\newtheorem{corollary}[prop]{Corollary}
\theoremstyle{definition}\newtheorem{theorem}[prop]{Theorem}
\theoremstyle{definition}\newtheorem{example}[prop]{Example}

\title{Independence equivalence classes of cycles}
\author{Boon Leong Ng}
\date{%
    National Institute of Education, Nanyang Technological University, Singapore\\
    \today
}

\begin{document}

\maketitle

\begin{abstract}
The independence equivalence class of a graph $G$ is the set of graphs that have the same independence polynomial as $G$. Beaton, Brown and Cameron~(2019) found the independence equivalence classes of even cycles, and raised the problem of finding the independence equivalence class of odd cycles. The problem is completely solved in this paper.
\end{abstract}

\section{Definitions and Introduction}
Let $G$ be a simple graph with vertex set $V(G)$ and edge set $E(G)$. A set of vertices $U\subseteq V(G)$ is said to be \emph{independent} if no two vertices in $U$ are adjacent. The \emph{independence number} $\alpha(G)$ of $G$ is the cardinality of the largest independent set of $G$. The \emph{independence polynomial} $I(G,x)$ of $G$ is given by $I(G,x) = \sum_{k=0}^{\alpha(G)} i_k(G) x^k$ where $i_k(G)$ is the number of independent subsets of $V(G)$ of cardinality $k$. A survey of results on independence polynomials can be found in~\cite{levit}.

Two graphs $G$ and $H$ are said to be \emph{independence equivalent} if $I(G,x) = I(H,x)$. The \emph{independence equivalence class} $\mathcal{I}(G)$ of a graph $G$ is the set of graphs which are independence equivalent to $G$. A graph $G$ is said to be \emph{independence unique} if $\mathcal{I}(G)=\{G\}$, that is, if $I(G,x) = I(H,x)$ implies that $G\cong H$ ($G$ is isomorphic to $H$).

In the study of graph polynomials, the problem of finding nonisomorphic graphs that have equivalent polynomials is one that naturally arises. An equivalence class analogous to the independence equivalence class can be defined for any graph polynomial. In the case of chromatic polynomials, this has a long history (see~\cite{bari,chao,chiaho,lauzhang} and Chapter~3 of~\cite{dong}). More recently, the study of domination polynomials has also led to results on domination equivalence~\cite{akbari,alikhanidom,beaton2,jahari}.

Returning to independence polynomials, the existence of graphs with the same independence polynomial was already noticed very early on by Wingard~\cite{wingard}. The following two results, on finding the independence polynomial of a graph in terms of that of its subgraphs, are very useful in showing that graphs are independence equivalent.
\begin{prop}\cite{hoedeli}\label{prop:delvert}
For any vertex $u\in V(G)$,
$$I(G, x) = I(G-u, x)+xI(G-N[u], x).$$
\end{prop}
\begin{prop}\cite{hoedeli}\label{prop:deledge}
For any edge $e=uv$ in $E(G)$,
$$I(G, x) = I(G-e, x)-x^2 I(G-(N(u)\cup N(v)), x).$$
\end{prop}
\begin{figure}
\centering
\begin{tikzpicture}
\filldraw(0,0) circle[radius=2pt]node[left]{$a$};
\filldraw(0,2) circle[radius=2pt];
\filldraw(1,1) circle[radius=2pt]node[above]{$b$};
\filldraw(2,1) circle[radius=2pt];
\node at (3,1) {$\cdots$};
\filldraw(4,1) circle[radius=2pt];
\draw(0,0)--(0,2);
\draw(0,0)--(1,1);
\draw(0,2)--(1,1);
\draw(1,1)--(2,1);
\draw(2,1)--(2.5,1);
\draw(3.5,1)--(4,1);
\node at (0.6,0.4) {$f$};
\end{tikzpicture}
\caption{The graph $D_n$ with $n$ vertices}\label{fig:dn}
\end{figure}
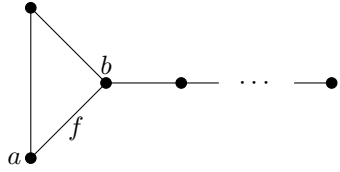
For example, the following family of independence equivalent graphs was observed by Chism~\cite{chism}.

\begin{prop}\cite{chism}\label{prop:cndn}
For $n\geq 4$, if $G$ is the cyclic graph $C_n$ and $H$ is the graph $D_n$ shown in Figure~\ref{fig:dn}, then $C_n$ and $D_n$ are independence equivalent.
\end{prop}

Oboudi~\cite{oboudi} and Beaton, Brown and Cameron~\cite{beaton} considered the question of whether other graphs existed in $\mathcal{I}(C_n)$, and showed the following:

\begin{prop}\cite{oboudi}\label{prop:connecteddn}
If $G$ is a connected graph in $\mathcal{I}(C_n)$, then $G\in\{C_n, D_n\}$.
\end{prop}

\begin{prop}\cite{beaton}\label{prop:oldcn}
\begin{enumerate}[label=(\roman*)]
    \item $\mathcal{I}(C_6)=\{C_6, D_6, K_4-e \cup K_2\}$ for $n=6$,
    \item $\mathcal{I}(C_n)=\{C_n,D_n\}$ for even $n\geq 4$, $n\neq 6$,
    \item $\mathcal{I}(C_n)=\{C_n,D_n\}$ for $n=p^k$ where $p\geq 5$ is prime and $k$ is a positive integer.
\end{enumerate}
\end{prop}

Beaton, Brown and Cameron~\cite{beaton} also made the following conjecture:

\begin{conj}
If $3\not\vert n$ and $n\geq 4$ is odd, then $G$ is independence equivalent to $C_n$ iff $G\in\{C_n,D_n\}$.
\end{conj}

In the case where $n$ is an odd multiple of $3$, Oboudi~\cite{oboudi} observed for $n=9$, and Beaton, Brown and Cameron~\cite{beaton} observed for $n=15$, that there exist graphs other than $D_n$ which are independence equivalent to $C_n$. A computer search carried out by Beaton, Brown and Cameron~\cite{beaton} showed that these were the only examples up to $n=31$, and they raised the problem of determining $\mathcal{I}(C_n)$ for larger values of $n$ which are odd multiples of $3$.

In this paper, we will prove Beaton, Brown and Cameron's conjecture, and solve their problem, in the following theorem.
\begin{theorem}\label{mytheorem}
For odd $n\geq 3$,
\begin{enumerate}[label=(\roman*)]
    \item $\mathcal{I}(C_3)=\{C_3\}$,
    \item $\mathcal{I}(C_0)=\{C_9, D_9, C_3 \cup G_a, C_3 \cup G_b, C_3 \cup G_c, C_3 \cup G_d\}$ (see Figure~\ref{fig:c9}; the case $C_3 \cup G_d$ seems to have been omitted by previous writers~\cite{oboudi,beaton}),
    \item $\mathcal{I}(C_{15})=\{C_{15}, D_{15}, C_3 \cup C_5 \cup G^\prime_a, C_3 \cup D_5 \cup G^\prime_a, C_3 \cup C_5 \cup G^\prime_b, C_3 \cup D_5 \cup G^\prime_b, C_3 \cup C_5 \cup G^\prime_c, C_3 \cup D_5 \cup G^\prime_c\}$ (see Figure~\ref{fig:c15}),
    \item $\mathcal{I}(C_n)=\{C_n,D_n\}$ for odd $n\geq 5$, except for $n=9$ and $n=15$,
\end{enumerate}
\end{theorem}

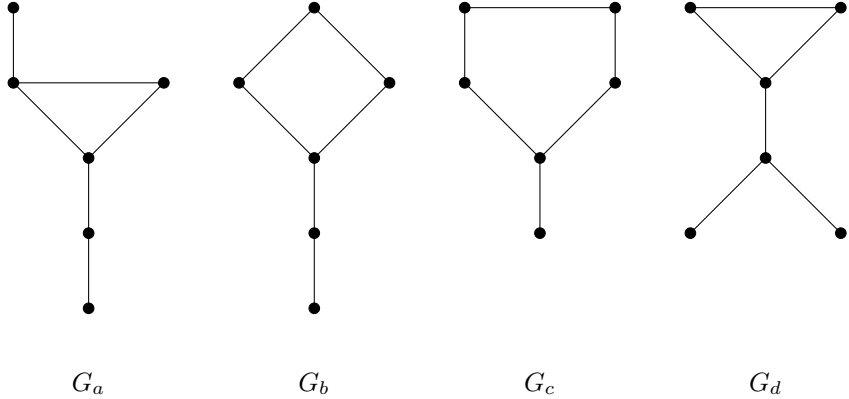
\begin{figure}
\centering
\begin{tikzpicture}
\filldraw(0,4) circle[radius=2pt];
\filldraw(0,3) circle[radius=2pt];
\filldraw(2,3) circle[radius=2pt];
\filldraw(1,2) circle[radius=2pt];
\filldraw(1,1) circle[radius=2pt];
\filldraw(1,0) circle[radius=2pt];
\draw(0,4)--(0,3);
\draw(0,3)--(2,3);
\draw(0,3)--(1,2);
\draw(2,3)--(1,2);
\draw(1,2)--(1,1);
\draw(1,1)--(1,0);
\node at (1,-1) {$G_a$};

\filldraw(4,4) circle[radius=2pt];
\filldraw(3,3) circle[radius=2pt];
\filldraw(5,3) circle[radius=2pt];
\filldraw(4,2) circle[radius=2pt];
\filldraw(4,1) circle[radius=2pt];
\filldraw(4,0) circle[radius=2pt];
\draw(4,4)--(3,3);
\draw(4,4)--(5,3);
\draw(3,3)--(4,2);
\draw(5,3)--(4,2);
\draw(4,2)--(4,1);
\draw(4,1)--(4,0);
\node at (4,-1) {$G_b$};

\filldraw(6,4) circle[radius=2pt];
\filldraw(8,4) circle[radius=2pt];
\filldraw(6,3) circle[radius=2pt];
\filldraw(8,3) circle[radius=2pt];
\filldraw(7,2) circle[radius=2pt];
\filldraw(7,1) circle[radius=2pt];
\draw(6,4)--(8,4);
\draw(6,4)--(6,3);
\draw(8,4)--(8,3);
\draw(6,3)--(7,2);
\draw(8,3)--(7,2);
\draw(7,2)--(7,1);
\node at (7,-1) {$G_c$};

\filldraw(9,4) circle[radius=2pt];
\filldraw(11,4) circle[radius=2pt];
\filldraw(10,3) circle[radius=2pt];
\filldraw(10,2) circle[radius=2pt];
\filldraw(9,1) circle[radius=2pt];
\filldraw(11,1) circle[radius=2pt];
\draw(9,4)--(11,4);
\draw(9,4)--(10,3);
\draw(11,4)--(10,3);
\draw(10,3)--(10,2);
\draw(10,2)--(9,1);
\draw(10,2)--(11,1);
\node at (10,-1) {$G_d$};
\end{tikzpicture}
\caption{Graphs $G_a$, $G_b$, $G_c$ and $G_d$ in Theorem~\ref{mytheorem}}\label{fig:c9}
\end{figure}

\begin{figure}
\centering
\begin{tikzpicture}
\filldraw(0,4) circle[radius=2pt];
\filldraw(0,3) circle[radius=2pt];
\filldraw(2,3) circle[radius=2pt];
\filldraw(1,2) circle[radius=2pt];
\filldraw(1,1) circle[radius=2pt];
\filldraw(1,0) circle[radius=2pt];
\filldraw(1,-1) circle[radius=2pt];
\draw(0,4)--(0,3);
\draw(0,3)--(2,3);
\draw(0,3)--(1,2);
\draw(2,3)--(1,2);
\draw(1,2)--(1,1);
\draw(1,1)--(1,0);
\draw(1,0)--(1,-1);
\node at (1,-2) {$G^\prime_a$};

\filldraw(5,4) circle[radius=2pt];
\filldraw(4,3) circle[radius=2pt];
\filldraw(6,3) circle[radius=2pt];
\filldraw(5,2) circle[radius=2pt];
\filldraw(5,1) circle[radius=2pt];
\filldraw(5,0) circle[radius=2pt];
\filldraw(5,-1) circle[radius=2pt];
\draw(5,4)--(4,3);
\draw(5,4)--(6,3);
\draw(4,3)--(5,2);
\draw(6,3)--(5,2);
\draw(5,2)--(5,1);
\draw(5,1)--(5,0);
\draw(5,0)--(5,-1);
\node at (5,-2) {$G^\prime_b$};

\filldraw(9,4) circle[radius=2pt];
\filldraw(8,3) circle[radius=2pt];
\filldraw(10,3) circle[radius=2pt];
\filldraw(8,2) circle[radius=2pt];
\filldraw(10,2) circle[radius=2pt];
\filldraw(9,1) circle[radius=2pt];
\filldraw(9,0) circle[radius=2pt];
\draw(9,4)--(8,3);
\draw(9,4)--(10,3);
\draw(8,3)--(8,2);
\draw(10,3)--(10,2);
\draw(8,2)--(9,1);
\draw(10,2)--(9,1);
\draw(9,1)--(9,0);
\node at (9,-2) {$G^\prime_c$};
\end{tikzpicture}
\caption{Graphs $G^\prime_a$, $G^\prime_b$ and $G^\prime_c$ in Theorem~\ref{mytheorem}}\label{fig:c15}
\end{figure}
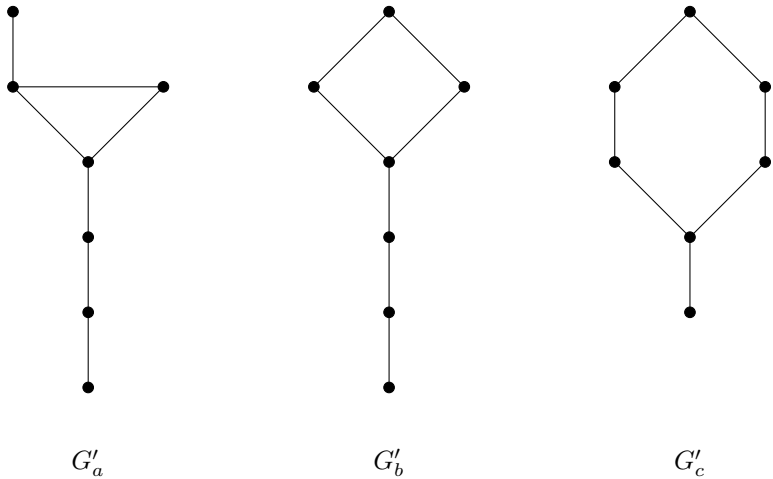

Proposition~\ref{prop:oldcn} and Thereom~\ref{mytheorem} together give us a complete characterisation of $\mathcal{I}(C_n)$ for all $n\geq 3$.

The proof of Theorem~\ref{mytheorem} will proceed as follows. Since Proposition~\ref{prop:connecteddn} completely characterises the connected graphs which are independence equivalent to $C_n$, we need only consider disconnected graphs. From the definition of the independence polynomial, it is clear that if $G$ is a disconnected graph with connected components $G_1,\ldots,G_r$, then
$$I(G,x)=\prod_{i=1}^r I(G_i,x).$$
In other words, the independence polynomial of a graph $G$ is the product of the independence polynomials of the connected components of $G$.

In Section~\ref{sec:factor}, we show how to factorise $I(C_n,x)$ over $\mathbb{Z}[x]$ for odd $n$, and investigate properties of the factors of $I(C_n,x)$. This in turn enables us, in Section~\ref{sec:structure}, to determine which graphs can be connected components of a graph $G\in\mathcal{I}(C_n)$. The case where $n$ is not a multiple of $3$ is addressed in Section~\ref{sec:notmult3}, where it is shown that $\mathcal{I}(C_n)=\{C_n,D_n\}$.

The more difficult case where $n$ is a multiple of $3$ is addressed in Section~\ref{sec:mult3}, where we seek graphs in $\mathcal{I}(C_n)$ that are not isomorphic to $C_n$ or $D_n$. Such graphs must be disconnected, and we show that one of the connected components must be $C_3$, and the other components must either be cycle graphs, or belong one of the families of graphs in Figure~\ref{fig:Gr}. Further consideration of the degree and coefficients of $I(C_n,x)$ narrows down the possible graphs to those listed in the statement of Theorem~\ref{mytheorem}.

\section{Factorisation of the Independence Polynomial of $C_n$}\label{sec:factor}
In this section, we factorise $I(C_n,x)$ over $\mathbb{Z}[x]$ to investigate properties of the factors.

The roots of the independence polynomials of cyclic graphs have been completely determined by Alikhani and Peng~\cite{alikhani}.
\begin{prop}\cite{alikhani}\label{prop:cyclicroot}
The roots of $I(C_n,x)$ are
$$c_i = -\frac{1}{2+2\cos\left(\frac{(2i-1)\pi}{n}\right)}$$
for $i = 1,2,\ldots,\left\lfloor\frac{n}{2}\right\rfloor$.
\end{prop}

Now, the minimal polynomials of $\cos(2\pi k / n)$ have been previously determined by Lehmer~\cite{lehmer} and Watkins and Zeitlin~\cite{watkins}.
\begin{prop}\cite{lehmer}\label{prop:cosroot1}
If $\gcd(k,n)=1$ and $n\geq 3$ then the minimal polynomial of $2\cos(2\pi k / n)$ has degree $\phi(n)/2$ and leading coefficient $1$, where $\phi(n)$ is Euler's totient function.
\end{prop}
Watkins and Zeitlin~\cite{watkins} give an explicit construction for this minimal polynomial in terms of Chebychev polynomials. 
\begin{prop}\cite{watkins}\label{prop:cosroot2}
The roots of the minimal polynomial of $2\cos(2\pi k / n)$ are precisely those values of $2\cos(2\pi k / n)$ for which $\gcd(k,n)=1$ and $1\leq k<n/2$.
\end{prop}

We are now in a position to find the minimal polynomials of $c_i$.

\begin{prop}\label{prop:cyclicdeg}
Let $n\geq 3$ be an odd integer. For positive integers $i$ where $\gcd(2i-1,n)=1$ and $1\leq i \leq \left\lfloor\frac{n}{2}\right\rfloor$, let $f(x)$ be the minimal polynomial of $c_i$ over $\mathbb{Z}[x]$. For positive integers $k$ where $\gcd(k,n)=1$ and $1\leq k<n/2$, let $g(x)$ be the minimal polynomial of $2\cos((2\pi k) / (2n))$ over $\mathbb{Z}[x]$. Then $f(x)$ and $g(x)$ are of the same degree $d=\phi(n)/2$. Furthermore, $f(x)\vert I(C_n,x)$.
\end{prop}

\begin{proof}
Let
$$x_i = 2\cos\left(\frac{2\pi(2i-1)}{2n}\right) = -\frac{2c_i+1}{c_i}.$$
Since $n$ is odd, $\gcd(2i-1,n)=1$ iff $\gcd(2i-1,2n)=1$. Therefore, the set of values taken by $x_i$ as $i$ varies is precisely those of $2\cos((2\pi k) / (2n))$ as $k$ varies.
\begin{eqnarray*}
& & \left\{x_i : \gcd(2i-1,n)=1\textrm{ and }1\leq i \leq \left\lfloor\frac{n}{2}\right\rfloor\right\} \\
& = & \left\{2\cos\left(\frac{2\pi k}{2n}\right) : \gcd(k,n)=1\textrm{ and }1\leq k<\frac{n}{2}\right\}
\end{eqnarray*}
From Proposition~\ref{prop:cosroot1}, $g(x)$ has degree $\phi(2n)/2=\phi(n)/2$ and leading coefficient 1. From Proposition~\ref{prop:cosroot2}, we now have
$$g(x)=\prod_{\substack{\gcd(2i-1,n)=1,\\1\leq i \leq \left\lfloor\frac{n}{2}\right\rfloor}} (x-x_i)$$
Since we know that $g(x)$ is an irreducible polynomial with integer coefficients and degree $d$, let
$$g(x) = \sum_{t=0}^d a_t x^t.$$
Now, we can translate $g(x)$ along the $x$-axis to obtain another irreducible polynomial $h(x)$ with integer coefficients and degree $d$, whose roots are $x_i + 2$.
$$g(x-2) = \sum_{t=0}^d a_t (x-2)^t = \sum_{t=0}^d b_t x^t = h(x).$$
Since $x_i + 2 = -1/c_i$,
$$0 = \sum_{t=0}^d a_t x_i^t = \sum_{t=0}^d b_t (x_i + 2)^t = \sum_{t=0}^d b_t \left(-\frac{1}{c_i}\right)^t = \left(-\frac{1}{c_i}\right)^d \sum_{t=0}^d b_t (-c_i)^{d-t}.$$
and therefore the $c_i$ are roots of a polynomial
$$ f(x) = \sum_{t=0}^d b_t (-x)^{d-t}. $$
Since $h(x)$ is irreducible, $f(x)$ is irreducible as well, and has degree $d$ and integer coefficients.

The roots of $f(x)$ are precisely the values of $c_i$ where $\gcd(2i-1,n)=1$ and $1\leq i \leq \left\lfloor\frac{n}{2}\right\rfloor$. As all these values of $c_i$ are also roots of $I(C_n,x)$, we see that $f(x)\vert I(C_n,x)$.
\end{proof}

We can therefore define a sequence of polynomials $f_n(x)$ for odd positive integers $n$.

\begin{definition}\label{def:fn}
For an odd positive integer $n$, the polynomial $f_n(x)$ is defined to be
$$ f(x) = \begin{cases}
  1 & n=1,\\
  k_n\displaystyle\prod_{\substack{{\gcd(2i-1,n)=1,}\\1{\leq i \leq \left\lfloor\frac{n}{2}\right\rfloor}}}(x-c_i) & \textrm{odd }n\geq 1,
\end{cases}
$$
where $k_n$ is an appropriate constant to make the coefficients of $f_n$ integers whose greatest common divisor is 1.\end{definition}
Then $f_n(x)$ is the minimal polynomial over $\mathbb{Z}[x]$ of $c_i$ where $\gcd(2i-1,n)=1$ and $1\leq i \leq \left\lfloor\frac{n}{2}\right\rfloor$ for odd $n>1$. The following proposition then follows immediately.
\begin{prop}\label{prop:cyclicfactor}
For odd $n>1$,
$$I(C_n,x)=\prod_{m\vert n} f_m(x).$$
\end{prop}
\begin{corollary}\label{cor:cyclicprime}
For an odd prime $p$, $f_p(x)=I(C_p,x)$.
\end{corollary}
\begin{corollary}\cite{beaton}\label{cor:cyclicdiv}
For an odd positive integer $n$ and $k\neq 1$, $k\vert n$ if and only if $I(C_k,x)\vert I(C_n,x)$.
\end{corollary}

\begin{example}
We will illustrate Proposition~\ref{prop:cyclicfactor} for $n=3,5,9,15$. (We omit explicitly writing the factor $f_1(x)=1$.)
\begin{enumerate}[label=(\roman*)]
    \item In the case $n=3$, we have
    $$f_3(x) = I(C_3,x) = 1+3x,$$
    an irreducible polynomial.
    
    Similarly, in the case $n=5$, we have
    $$f_5(x) = I(C_5,x) = I(D_5,x) = 1+5x+5x^2.$$
    \item In the case $n=9$, we have
    \begin{eqnarray*}
    I(C_9,x) & = & 1+ 9x + 27x^2+ 30x^3 + 9x^4\\
             & = &\underbrace{(1+3x)}_{f_3(x)}\underbrace{(1 + 6x + 9x^2 + 3x^3)}_{f_9(x)}.
    \end{eqnarray*}
    Note that for the four graphs in Figure~\ref{fig:c9}, $$I(G_a,x)=I(G_b,x)=I(G_c,x)=I(G_d,x)=f_9(x).$$
    \item In the case $n=15$, we have
    \begin{eqnarray*}
    I(C_{15},x) & = & 1 + 15x + 90x^2 + 275x^3 + 450x^4 + 378x^5 + 140x^6 + 15x^7\\
    & = &\underbrace{(1+3x)}_{f_3(x)}\underbrace{(1+5x+5x^2)}_{f_5(x)}\underbrace{(1+7x+14x^2+8x^3+x^4)}_{f_{15}(x)}.
    \end{eqnarray*}
    Note that for the three graphs in Figure~\ref{fig:c15}, $$I(G^\prime_a,x)=I(G^\prime_b,x)=I(G^\prime_c,x)=f_{15}(x).$$
\end{enumerate}
\end{example}

We now show how this factorisation of $I(C_n,x)$ gives us information about graphs which are independence equivalent to it.

\section{Structure of a graph which is independence equivalent to $C_n$}\label{sec:structure}

\begin{definition}\cite{beaton}
A polynomial $p(x)=\sum_{i=0}^n p_i x^i$, where the $p_i$ are all non-negative integers, is unicyclic if $p_0 = 1$ and $p_2 = \binom{p_1}{2}-p_1$.
\end{definition}

\begin{corollary}\cite{beaton}\label{cor:unicyclic}
Let $G$ be a connected graph. Then $G$ has an independence polynomial that is unicyclic iff $G$ is unicyclic.
\end{corollary}

\begin{prop}\cite{beaton}\label{prop:unicyclicprod}
Suppose that $p(x)=q(x)r(x)$.
\begin{enumerate}[label=(\roman*)]
    \item If $q(x)$ and $r(x)$ are unicyclic then $p(x)$ is unicyclic.
    \item If $p(x)$ and $q(x)$ are unicyclic then $r(x)$ is unicyclic.
\end{enumerate}
\end{prop}

\begin{prop}\label{prop:unicyclicf}
The polynomials $f_n(x)$ for odd $n\geq 3$ are all unicyclic.
\end{prop}
\begin{proof}
Suppose that $n$ has prime factorisation $p_1^{r_1}p_2^{r_2}\ldots p_j^{r_j}$. We use induction on $s_n = r_1 + \ldots + r_j$. If $s_n=1$ then $n$ is prime, and therefore, by Corollary~\ref{cor:cyclicprime}, $f_n(x)=I(C_n,x)$ which is unicyclic. Suppose it holds for all odd $n\geq 3$ where $1\leq s_n < s$.

Let $n\geq 3$ be odd where $s_n = s\geq 2$. Now
$$f_n(x)=\frac{I(C_n,x)}{\displaystyle\prod_{m\vert n, m < n}f_m(x)}.$$
$I(C_n,x)$ is unicyclic and, for all $m$ such that $m\vert n$, $m<n$, we have $s_m < s$ so $f_m(x)$ is unicyclic by our induction hypothesis. Therefore, by Proposition~\ref{prop:unicyclicprod}(ii), $f_n(x)$ is also unicyclic.
\end{proof}

\begin{prop}\label{prop:unicyclicsubgraph}
If $G\in\mathcal{I}(C_n)$, where $n$ is odd, and $G$ is the disjoint union of connected graphs $G_1, G_2, \ldots, G_r$, then $G_1, G_2, \ldots, G_r$ are all unicyclic.
\end{prop}

\begin{proof}
For each $i\in\{1,\ldots,r\}$, $I(G_i,x)$ is a factor of $I(C_n,x)$. Therefore, it must be a product of some polynomials $f_m(x)$ where $m\vert n$. Since each of the $f_m(x)$ are unicyclic, their product $I(G_i,x)$ is unicyclic by Proposition~\ref{prop:unicyclicprod}(i), and hence $G_i$ is unicyclic by Corollary~\ref{cor:unicyclic}.
\end{proof}

Therefore, if $G\in\mathcal{I}(C_n)$, where $n$ is odd, then $G=D_n$ or $G$ is the disjoint union of a set of unicyclic graphs.

We will make use of the following notation and propositions from Beaton, Brown and Cameron~\cite{beaton}.

\begin{definition}
Let $n\geq 4$ be a positive integer and $G$ be a graph. Let $n_G(C_3)$ be the number of triangles in $G$, and $g_i$ denote the number of vertices in $G$ of degree~$i$.
\end{definition}

\begin{prop}\cite{beaton}\label{prop:degreesum}
If $G\in\mathcal{I}(C_n)$, where $n\geq 4$, then
\begin{enumerate}[label=(\roman*)]
    \item $\displaystyle\sum_{i=0}^{n-1}g_i = n$,
    \item $\displaystyle\sum_{i=1}^{n-1}i\cdot g_i = 2n$,
    \item $\displaystyle\sum_{i=2}^{n-1}\binom{i}{2}g_i = n + n_G(C_3)$,
    \item $n_G(C_3)\geq g_0 + \displaystyle\sum_{i=3}^{n-1} g_i$, so that the number of vertices in $G$ not of degree~1 or~2 is $n(C_3)$.
\end{enumerate}
\end{prop}

Actually, Beaton, Brown and Cameron proved a stronger statement than Proposition~\ref{prop:degreesum}(iv).
\begin{prop}\label{prop:degreesum2}
If $G\in\mathcal{I}(C_n)$, where $n\geq 4$, then
$$n_G(C_3)= \sum_{i=3}^{n-1}\binom{i-1}{2}g_i$$
\end{prop}
\begin{proof}
From Proposition~\ref{prop:degreesum}, if we add (i) and (iii) and subtract (ii), we get
\begin{eqnarray*}
n_G(C_3) & = & \sum_{i=0}^{n-1}g_i+\sum_{i=2}^{n-1}\binom{i}{2}g_i-\sum_{i=1}^{n-1}i\cdot g_i\\
         & = & g_0+\sum_{i=3}^{n-1}\left(\binom{i}{2}-i+1\right)g_i \\
         & = & g_0+\sum_{i=3}^{n-1}\binom{i-1}{2}g_i.
\end{eqnarray*}
Now $g_0$ is the number of isolated vertices. However, if $G$ has an isolated vertex, then $I(K_1,x)=1+x$ would be a factor of $I(G,x)$, so $-1$ would be a root of $I(G,x)=I(C_n,x)$. However, we can see from Proposition~\ref{prop:cyclicroot} that this is not possible. Hence $g_0=0$ and the result follows.
\end{proof}

We now try to elucidate the structure of a graph $G\in\mathcal{I}(C_n)$ for odd $n\geq 5$. If $G$ is not $C_n$ or $D_n$ then, by Proposition~\ref{prop:connecteddn}, $G$ is not connected, and by Proposition~\ref{prop:unicyclicsubgraph}, $G=\bigcup_{i=1}^r G_i$ where the $G_i$ are connected unicyclic graphs.

Since $I(C_n,x)$ does not have repeated roots, at most one of the connected components can be $C_3$. If any of the other connected components contain a triangle, then at least one of the vertices of the triangle must have degree~3 or larger.

\begin{prop}\label{prop:maxdegree3}
If $G\in\mathcal{I}(C_n)$, where $n\geq 4$, then the maximum degree $\Delta(G)$ of $G$ is at most~$3$ and
$$n_G(C_3)=g_3.$$
\end{prop}
\begin{proof}
From Proposition~\ref{prop:degreesum2},
$$n_G(C_3)=g_3+3g_4+6g_5+\ldots+\binom{n-2}{2}g_{n-1}.$$
Let $h=g_4+\ldots+g_{n-1}>0$. Then $n_G(C_3)\geq g_3+3h$. Since every triangle, except at most one, has a vertex of degree~3 or larger, we see that $g_3+h$ must be at least as large as the number of triangles less one, that is, $g_3+h\geq n_G(C_3)-1$. Since we have $g_3+h\geq g_3+3h-1$ and $h$ is a nonnegative integer, the only possibility is that $h=0$, so $g_4=\ldots=g_{n-1}=0$ and the maximum degree of $G$ is~$3$. This also implies that $n_G(C_3)=g_3$.
\end{proof}

\begin{prop}\label{prop:unioncn}
If $G\in\mathcal{I}(C_n)$, where $n\geq 4$, and $G$ does not contain $C_3$ as a connected component, then each connected component of $G$ is either a cycle or a graph of the form $D_m$ for some values of $m$.
\end{prop}
\begin{proof}
Let $G_i$ be one of the connected components of $G$. From Proposition~\ref{prop:maxdegree3}, the maximum degree of~$G$ is at most~$3$. Since $G_i$ is unicyclic, if the maximum degree of $G_i$ is~$2$, then $G_i$ must be a cycle.

If the maximum degree of $G_i$ is~$3$, then $G_i$ must contain a triangle. Otherwise, $G$ would need to have $C_3$ as a connected component so that $n_G(C_3)=g_3$. But this also implies that there must be exactly one vertex of degree~3 per triangle, so there is only one such vertex in $G_i$. Therefore, $G_i$ must be isomorphic to $D_m$ for some value of $m$.

Therefore, each connected component of $G$ is either a cycle or a graph of the form $D_m$ for some values of $m$.
\end{proof}

\section{The case where $n$ is not a multiple of $3$}\label{sec:notmult3}

\begin{prop}\label{prop:notmult3}
If $G\in\mathcal{I}(C_n)$ where $n$ is an odd positive number that is not a multiple of~3, then $G=C_n$ or $G=D_n$.
\end{prop}
\begin{proof}
If $G$ is not $C_n$ or $D_n$ then $G$ is not connected, and by Proposition~\ref{prop:unicyclicsubgraph}, $G=\bigcup_{i=1}^r G_i$ where the $G_i$ are connected unicyclic graphs. Since $n$ is not a multiple of $3$, none of the $G_i$ can be $C_3$, by Corollary~\ref{cor:cyclicdiv}.

By Proposition~\ref{prop:unioncn}, $G$ is a disjoint union of cyclic graphs and $D_m$'s. But each $D_m$ is independence equivalent to $C_m$, hence $G$ is independence equivalent to a disjoint union of cyclic graphs. In other words, $I(G,x)$ is a product of independence polynomials of cyclic graphs:

$$I(G,x)=\prod_{i=1}^r I(C_{n_i},x).$$

Now I claim that $\gcd(n_i,n_j)=1$ for all $i,j\in\{1,\ldots\,r\}$. Suppose this was not the case, and that for some $i,j\in\{1,\ldots\,r\}$, $\gcd(n_i,n_j)=d>1$. Then, from Proposition~\ref{prop:cyclicroot},
$$-\frac{1}{2+2\cos\left(\frac{\pi}{d}\right)}=-\frac{1}{2+2\cos\left(\frac{(n_i/d)\pi}{n_i}\right)}=-\frac{1}{2+2\cos\left(\frac{(n_j/d)\pi}{n_j}\right)}$$
is a root of both $I(C_{n_i},x)$ and $I(C_{n_j},x)$, and hence a repeated root of $I(G,x)$. However, also from Proposition~\ref{prop:cyclicroot}, none of the roots of $I(G,x)=I(C_n,x)$ are repeated, which is a contradiction. Hence, $\gcd(n_i,n_j)=1$ for all $i,j\in\{1,\ldots\,r\}$.

Now, from Proposition~\ref{prop:cyclicfactor},
$$I(G,x)=\prod_{i=1}^r \prod_{m\vert n_i} f_m(x).$$
For some $i,j\in\{1,\ldots\,r\}$, note that $f_{n_i}(x)$ and $f_{n_j}(x)$ are included among the factors, but since $\gcd(n_i,n_j)=1$, $f_{n_in_j}(x)$ is not.

We also have
$$I(G,x)=I(C_n,x)=\prod_{m\vert n} f_m(x).$$
Therefore, since $f_{n_i}(x)$ and $f_{n_j}(x)$ are factors of $I(C_n,x)$, it follows that $n_i\vert n$ and $n_j\vert n$, and since $\gcd(n_i,n_j)=1$, $n_in_j\vert n$, so  $f_{n_in_j}(x)$ is also a factor of $I(G,x)$. This is a contradiction.

Hence, $G$ must be a connected graph and so can only be $C_n$ or $D_n$.
\end{proof}
\section{The case where $n$ is a multiple of $3$}\label{sec:mult3}
If $n$ is a multiple of $3$, and if $G$ is a disconnected graph which is independence equivalent to $C_n$, then it turns out that $C_3$ must be exactly one of the connected components of $G$. If this is not the case, then the proof of Proposition~\ref{prop:notmult3} would show that $G$ must be a connected graph. On the other hand, since the independence polynomial of $C_n$ does not have repeated roots, at most one of the connected components can be $C_3$. If any of the other connected components contain a triangle, then at least one of the vertices of the triangle must have degree~3 or larger.

Let $G=\bigcup_{i=1}^r G_i$ where the $G_i$ are connected unicyclic graphs. Let $g_{3,i}$ be the number of vertices of degree~3 in $G_i$, and let $G_1=C_3$. Now $n_{G_i}(C_3)$ is either~$0$ or~$1$, and if $n_{G_i}(C_3)=1$ then $g_{3,i}\geq 1$ for $i \in \{2,\ldots, r\}$. Therefore, $g_{3,i}\geq n_{G_i}(C_3)$ for all $i \in \{2,\ldots, r\}$. But we have
$$1+\sum_{i=2}^r n_{G_i}(C_3) = n_G(C_3) = g_3 = \sum_{i=2}^r g_{3,i}$$
and so $g_{3,i} = n_{G_i}(C_3)$ for all values of $i \in \{2,\ldots, r\}$ except one (without loss of generality, let it be $r$), and in that case, $g_{3,r}= n_{G_r}(C_3)+1$.

For those values of $i$ where $g_{3,i} = n_{G_i}(C_3)$, we have either $g_{3,i} = n_{G_i}(C_3)=0$, which means $G_i=C_m$ for some $m\geq 4$, or $g_{3,i} = n_{G_i}(C_3)=1$, which, by an argument in the proof of Proposition~\ref{prop:unioncn}, means that $G_i=D_m$ for some $m\geq 4$. Since $C_m$ and $D_m$ are independence equivalent, we will not consider the case where $G_i=D_m$, and instead replace all $D_m$ with $C_m$. Since $I(C_m,x)\vert I(C_n,x)$, by Proposition~\ref{cor:cyclicdiv}, $m\vert n$. In particular, $m$ is odd.

As for $G_r$, where $g_{3,r}= n_{G_r}(C_3)+1$, we have the following cases (see Figure~\ref{fig:Gr}), where $m_1,m_2,m_3$ are positive integers unless otherwise stated.
\begin{enumerate}
    \item $g_{3,r}=1$ and $n_{G_r}(C_3)=0$. In this case, the only possibility is $E_{m_1,m_2}$, a graph where a pendant path $P_{m_2}$ is attached to a cyclic graph $C_{m_1+3}$.
    \item $g_{3,r}=2$ and $n_{G_r}(C_3)=1$. This leads us to two subcases:
    \begin{enumerate}
        \item If both vertices of degree~3 are on the triangle, then we have $A_{m_1,m_2}$, a graph where pendant paths $P_{m_1}$ and $P_{m_2}$ are attached to two distinct vertices of a $C_3$.
        \item If only one vertex of degree~3 is on the triangle, then we have $B_{m_1,m_2,m_3}$. It is possible that $m_1=0$, in which case the two vertices of degree~3 are adjacent.
    \end{enumerate}
\end{enumerate}
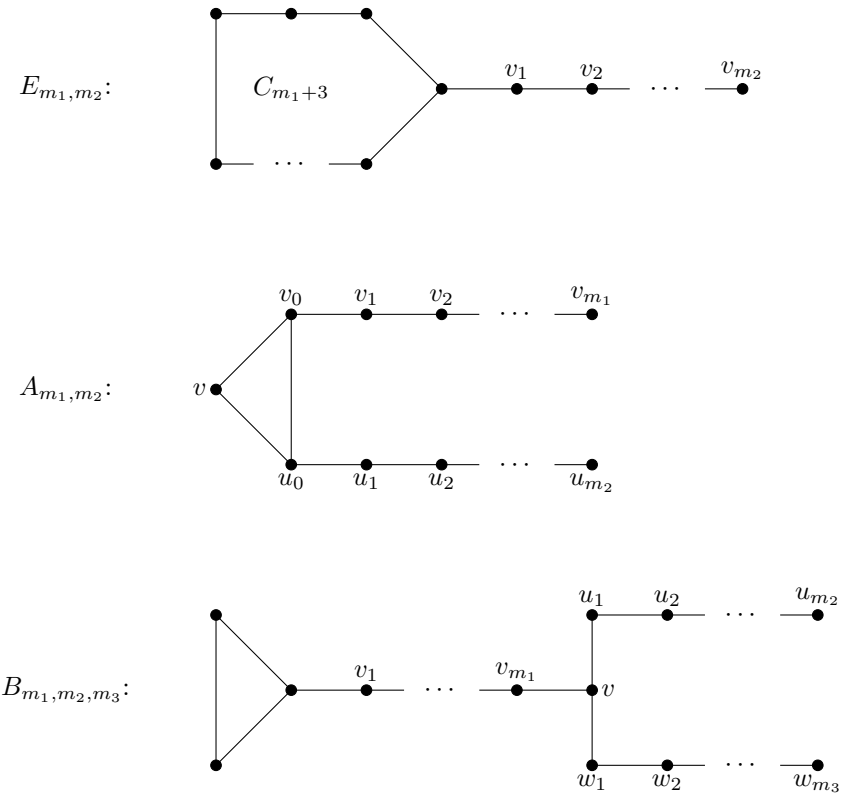
\begin{figure}
\centering
\begin{tikzpicture}
\node at (-2,9) {$E_{m_1,m_2}$:};
\filldraw(0,10) circle[radius=2pt];
\filldraw(1,10) circle[radius=2pt];
\filldraw(2,10) circle[radius=2pt];
\filldraw(0,8) circle[radius=2pt];
\node at (1,8) {$\cdots$};
\node at (1,9) {$C_{m_1+3}$};
\filldraw(2,8) circle[radius=2pt];
\filldraw(3,9) circle[radius=2pt];
\filldraw(4,9) circle[radius=2pt]node[above]{$v_1$};
\filldraw(5,9) circle[radius=2pt]node[above]{$v_2$};
\node at (6,9) {$\cdots$};
\filldraw(7,9) circle[radius=2pt]node[above]{$v_{m_2}$};
\draw(0,10)--(1,10);
\draw(1,10)--(2,10);
\draw(0,10)--(0,8);
\draw(0,8)--(0.5,8);
\draw(1.5,8)--(2,8);
\draw(2,10)--(3,9);
\draw(2,8)--(3,9);
\draw(3,9)--(4,9);
\draw(4,9)--(5,9);
\draw(5,9)--(5.5,9);
\draw(6.5,9)--(7,9);

\node at (-2,5) {$A_{m_1,m_2}$:};
\filldraw(0,5) circle[radius=2pt]node[left]{$v$};
\filldraw(1,6) circle[radius=2pt]node[above]{$v_0$};
\filldraw(2,6) circle[radius=2pt]node[above]{$v_1$};
\filldraw(3,6) circle[radius=2pt]node[above]{$v_2$};
\node at (4,6) {$\cdots$};
\filldraw(5,6) circle[radius=2pt]node[above]{$v_{m_1}$};
\filldraw(1,4) circle[radius=2pt]node[below]{$u_0$};
\filldraw(2,4) circle[radius=2pt]node[below]{$u_1$};
\filldraw(3,4) circle[radius=2pt]node[below]{$u_2$};
\node at (4,4) {$\cdots$};
\filldraw(5,4) circle[radius=2pt]node[below]{$u_{m_2}$};
\draw(0,5)--(1,6);
\draw(0,5)--(1,4);
\draw(1,6)--(2,6);
\draw(2,6)--(3,6);
\draw(3,6)--(3.5,6);
\draw(4.5,6)--(5,6);
\draw(1,6)--(1,4);
\draw(1,4)--(2,4);
\draw(2,4)--(3,4);
\draw(3,4)--(3.5,4);
\draw(4.5,4)--(5,4);

\node at (-2,1) {$B_{m_1,m_2,m_3}$:};
\filldraw(0,0) circle[radius=2pt];
\filldraw(0,2) circle[radius=2pt];
\filldraw(1,1) circle[radius=2pt];
\filldraw(2,1) circle[radius=2pt]node[above]{$v_1$};
\node at (3,1) {$\cdots$};
\filldraw(4,1) circle[radius=2pt]node[above]{$v_{m_1}$};
\filldraw(5,1) circle[radius=2pt]node[right]{$v$};
\filldraw(5,2) circle[radius=2pt]node[above]{$u_1$};
\filldraw(6,2) circle[radius=2pt]node[above]{$u_2$};
\node at (7,2) {$\cdots$};
\filldraw(8,2) circle[radius=2pt]node[above]{$u_{m_2}$};
\filldraw(5,0) circle[radius=2pt]node[below]{$w_1$};
\filldraw(6,0) circle[radius=2pt]node[below]{$w_2$};
\node at (7,0) {$\cdots$};
\filldraw(8,0) circle[radius=2pt]node[below]{$w_{m_3}$};
\draw(0,0)--(0,2);
\draw(0,0)--(1,1);
\draw(0,2)--(1,1);
\draw(1,1)--(2,1);
\draw(2,1)--(2.5,1);
\draw(3.5,1)--(4,1);
\draw(4,1)--(5,1);
\draw(5,1)--(5,2);
\draw(5,2)--(6,2);
\draw(6,2)--(6.5,2);
\draw(7.5,2)--(8,2);
\draw(5,1)--(5,0);
\draw(5,0)--(6,0);
\draw(6,0)--(6.5,0);
\draw(7.5,0)--(8,0);
\end{tikzpicture}
\caption{Candidate graphs for $G_r$}\label{fig:Gr}
\end{figure}

It turns out that the graph $A_{m_1,m_2}$ is independence equivalent to both $E_{m_1,m_2}$ and $E_{m_2,m_1}$. This can be observed using Proposition~\ref{prop:delvert}. Deleting vertices $v_1$ from $E_{m_1,m_2}$ and $u_1$ from $A_{m_1,m_2}$ give isomorphic graphs, and deleting neighbourhoods $N[v_1]$ from $E_{m_1,m_2}$ and $N[u_1]$ from $A_{m_1,m_2}$ also give isomorphic graphs. Therefore $A_{m_1,m_2}$ is independence equivalent to $E_{m_1,m_2}$. Similarly, deleting vertices $v_1$ from $E_{m_2,m_1}$ and $v_1$ from $A_{m_1,m_2}$ gives isomorphic graphs, and deleting neighbourhoods $N[v_1]$ from $E_{m_2,m_1}$ and $N[v_1]$ from $A_{m_1,m_2}$ also give isomorphic graphs, so $A_{m_1,m_2}$ is independence equivalent to $E_{m_2,m_1}$. Therefore, we do not need to consider graphs $E_{m_1,m_2}$ for the rest of this section, as they can be replaced with graphs $A_{m_1,m_2}$.

As $G$ is independence equivalent to $C_n$, their independence numbers are equal so $\alpha(G)=\alpha(C_n)=\frac{1}{2}(n-1)$ since $n$ is odd. As $G=\bigcup_{i=1}^r G_i$ and all the $G_i$ are odd cycles except for $G_r$, we have
$$\alpha(G)=\sum_{i=1}^r \alpha(G_i) =\alpha(G_r)+\sum_{i=1}^{r-1}\frac{|V(G_i)|-1}{2}=\alpha(G_r)-\frac{r-1}{2}+\sum_{i=1}^{r-1}\frac{|V(G_i)|}{2}.$$
Equating this to $\frac{1}{2}(n-1)$ gives us
\begin{equation}\label{eq:boundr}
2\alpha(G_r)=n+r-2-\sum_{i=1}^{r-1}V(G_i)|=|V(G_r)|+r-2.
\end{equation}
Since $G_r$ is of the form $A_{m_1,m_2}$ or $B_{m_1,m_2,m_3}$, we can calculate the independence numbers of both graphs to determine $\alpha(G_r)$ and obtain a bound on $r$.

\begin{lemma}
The independence number $\alpha(A_{m_1,m_2})$ is
\begin{itemize}
    \item $\frac{1}{2}(m_1+m_2+4)=\frac{1}{2}(|V(A_{m_1,m_2})|+1)$ if $m_1$ and $m_2$ are both odd,
    \item $\frac{1}{2}(m_1+m_2+3)=\frac{1}{2}|V(A_{m_1,m_2})|$ if $m_1$ and $m_2$ have different parity,
    \item $\frac{1}{2}(m_1+m_2+2)=\frac{1}{2}(|V(A_{m_1,m_2})|-1)$ if $m_1$ and $m_2$ are both even.
\end{itemize}
Also, the independence number $\alpha(B_{m_1,m_2,m_3})$ is
\begin{itemize}
    \item $\frac{1}{2}(m_1+m_2+m_3+5)=\frac{1}{2}(|V(B_{m_1,m_2,m_3})|+1$ if $m_1,m_2,m_3$ are all odd,
    \item $\frac{1}{2}(m_1+m_2+m_3+4)=\frac{1}{2}|V(B_{m_1,m_2,m_3})|$ if
    \begin{itemize}
        \item $m_1$ is odd and $m_2$ and $m_3$ have different parity,
        \item $m_1$ is even and $m_2$ and $m_3$ are both odd,
        \item $m_1,m_2,m_3$ are all even,
    \end{itemize}
    that is, either none or exactly two of $m_1,m_2,m_3$ are odd,
    \item $\frac{1}{2}(m_1+m_2+m_3+3)=\frac{1}{2}(|V(B_{m_1,m_2,m_3})|-1)$ if
    \begin{itemize}
        \item $m_1$ is odd and $m_2$ and $m_3$ are both even,
        \item $m_1$ is even and $m_2$ and $m_3$ have different parity,
    \end{itemize}
    that is, exactly one of $m_1,m_2,m_3$ is odd.
\end{itemize}
\end{lemma}
\begin{proof}
Using Proposition~\ref{prop:delvert} and the vertices labelled~$v$ in Figure~\ref{fig:Gr}, we find that
$$I(A_{m_1,m_2},x)=I(P_{m_1+m_2+2},x)+xI(P_{m_1},x)I(P_{m_2},x)$$
and
$$
I(B_{m_1,m_2,m_3},x)=\begin{cases}
  \begin{array}{l}
     I(C_{m_1+3},x)I(P_{m_2},x)I(P_{m_3},x)\\
      \qquad+xI(C_{m_1+2},x)I(P_{m_2-1},x)I(P_{m_3-1},x)
  \end{array} & m_1 \geq 1, \\
  \begin{array}{l}
     I(C_{m_1+3},x)I(P_{m_2},x)I(P_{m_3},x)\\
      \qquad+xI(P_2,x)I(P_{m_2-1},x)I(P_{m_3-1},x)
  \end{array} & m_1 =0. \\ 
\end{cases}
$$
Since we know that the degree of $I(P_m,x)$ is $\frac{1}{2}(m+1)$ when $m$ is odd and $\frac{1}{2}m$ when $m$ is even, and the degree of $I(C_m,x)$ is $\frac{1}{2}(m-1)$ when $m$ is odd and $\frac{1}{2}m$ when $m$ is even, we can calculate the degree of $I(A_{m_1,m_2},x)$ and $I(B_{m_1,m_2,m_3},x)$ for each of the cases.
\end{proof}
Substituting each of these possibilities into Equation~\ref{eq:boundr}, we find that $r\in\{1,2,3\}$. However, as $G$ is not a connected graph, $r\neq 1$. This leaves us with the following cases:
\begin{itemize}
    \item $r=2$: Then
    \begin{itemize}
        \item $G=C_3 \cup A_{m_1,m_2}$, where $m_1$ and $m_2$ are of different parity
        \item $G=C_3 \cup B_{m_1,m_2,m_3}$ where either none or exactly two of $m_1,m_2,m_3$ are odd.
    \end{itemize}
    \item $r=3$: Then
    \begin{itemize}
        \item $G=C_3 \cup C_m \cup A_{m_1,m_2}$, where $m_1$ and $m_2$ are both odd, or
        \item $G=C_3 \cup C_m \cup B_{m_1,m_2,m_3}$ where $m_1,m_2,m_3$ are all odd.
    \end{itemize}
    Note that $m\vert n$ (from Corollary~\ref{cor:cyclicdiv}), and $3\nmid m$, since from Proposition~\ref{prop:cyclicroot}, $I(C_n,x)$ does not have repeated roots. Note that we can always replace $C_m$ by $D_m$.
\end{itemize}

Thus far, we have used the fact that the independence coefficients $i_k(G)$ for $k\in\{0,1,2,3\}$ are equal to $i_k(C_n)$ to obtain information about the structure of $G$. We next try to make use of the next coefficient $i_4(G)$. Now, the coefficients of the independence polynomial of the cyclic graph $C_n$ and the path $P_n$ have been found by Hopkins and Staton~\cite{hopkins}.
\begin{lemma}\cite{hopkins}\label{lemma:cycliccoeff}
Let $i_k(G)$ be the number of independent sets of cardinality $k$ in $G$. Then
$$i_k(C_n)=\frac{n}{k}\binom{n-k-1}{k-1}$$
and
$$i_k(P_n)=\binom{n-k+1}{k}.$$
\end{lemma}

We also need the following:
\begin{lemma}\label{lem:i4}
If $G$ is a graph such that $I(G,x)=I(C_n,x)$ where $n \geq 5$ is odd. Then
\begin{eqnarray*}
\frac{n(3n-11)}{2} & = & e_2(G) + n_G(P_3 \cup K_1) - n_G(C_3 \cup K_1) \\ 
& &- n_G(P_4) - n_G(K_{1,3}) + n_G(D_4) + n_G(C_4).
\end{eqnarray*}
where $e_2(G)$ is the number of matchings of size~$2$ in $G$ and $n_G(H)$ is the number of subgraphs in $G$ (not necessarily induced) which are isomorphic to $H$.
\end{lemma}
\begin{proof}
We try to find an expression for $i_4(G)$ using the Principle of Inclusion and Exclusion. Any subset of 4 vertices would induce one of the subgraphs in the table below. As each connected component of $G$ is unicyclic, we do not need to consider subgraphs with more than $4$ edges as they would contain more than one cycle.
\begin{center}
\begin{tabular}{c|c}
    Number of edges $|E(H)|$ & Subgraphs $H$\\
    \hline\hline
    $0$ & $\overline{K}_4$ \\
    \hline
    $1$ & $P_2\cup\overline{K}_2$ \\
    \hline
    $2$ & $2P_2$ and $P_3\cup K_1$ \\
    \hline
    $3$ & $C_3\cup K_1$, $P_4$ and $K_{1,3}$ \\
    \hline
    $4$ & $D_4$ and $C_4$
\end{tabular}
\end{center}
Let $\mathcal{G}_4$ be the collection of graphs listed in the table above.

For each edge $e \in E(G)$, let $X_e$ (respectively, $\overline{X_e}$) be the set of induced subgraphs of $G$ with $4$ vertices and which contain (do not contain) $e$ as an edge. Then
\begin{eqnarray*}
i_4(G) & = & \left| \bigcap\limits_{e\in E(G)}\overline{X_e}\right| \\
& = & \sum_{Y\subseteq E(G)} (-1)^{|Y|}\left|\bigcap\limits_{e\in Y}X_e \right|,
\end{eqnarray*}
by the Principle of Inclusion and Exclusion.
For each $k$, $1\leq k\leq |E(G)|$, consider a set $Y\subseteq E(G)$ such that $|Y| = k$. Then $\bigcap\limits_{e\in Y}X_e$ is the collection of graphs $H$ which have $4$ vertices and edge set $E(H) = Y$. Therefore,
$$\sum_{\substack{H\in\mathcal{G}_4,\\|E(H)|=k}} n_G (H) = \sum_{\substack{Y\subseteq E(H), \\|Y|=k}} \left|\bigcap\limits_{e\in Y}X_e \right|.$$
Substituting this into the above expression for $i_4(G)$ gives us
$$i_4(G) = \sum_{H\in\mathcal{G}_4} (-1)^{|E(H)|} n_{G}(H).$$
Expanding this summation using the values of $n_G(H)$ from the table above gives us
\begin{eqnarray*}
i_4(G) & = & \binom{n}{4}-n\binom{n-2}{2}\\
       &   & + e_2(G) + n_G(P_3 \cup K_1)\\
       &   & - n_G(C_3 \cup K_1) - n_G(P_4) - n_G(K_{1,3})\\
       &   & + n_G(D_4)+ n_G(C_4)
\end{eqnarray*}
because
$$n_G(\overline{K}_4) = \binom{n}{4},\qquad n_G(P_2\cup\overline{K}_2) = n\binom{n-2}{2},\qquad n_G(2P_2) = e_2(G).$$
On the other hand, from Lemma~\ref{lemma:cycliccoeff},
$$i_4(G) = i_4(C_n) = \frac{n(n-5)(n-6)(n-7)}{24}.$$
Comparing the two expressions for $i_4(G)$ gives us the required result.
\end{proof}

\subsection{The subcase $G=C_3 \cup A_{m_1,m_2}$}\label{sec:notmult3-1}

We tabulate the number of each subgraph with~2 to~4 vertices for $G=C_3\cup A_{m_1,m_2}$. (Here, $n=|V(G)|=|V(A_{m_1,m_2})|+3$.)

We will do it in detail for this subcase, and leave the remaining subcases for the reader to fill in. The vertex labels below refer to those in Figure~\ref{fig:Gr}.

\begin{itemize}
    \item $2P_2$: This is the number of matchings of size~$2$ in $C_3\cup A_{m_1,m_2}$. In such a matching, either the two edges are both in $A_{m_1,m_2}$ or one edge is in $C_3$ and one edge is in $A_{m_1,m_2}$. There are $e_2(A_{m_1,m_2})$ matchings in the first case. We find $e_2(A_{m_1,m_2})$ by subtracting the number of pairs of adjacent edges from $\binom{n-3}{2}$. Now there are $n-7$ vertices of degree~$2$ (namely, $v$, $v_1$ to $v_{m_1-1}$ and $u_1$ to $u_{m_2-1}$) in $A_{m_1,m_2}$ that correspond to a pair of adjacent edges, and two vertices of degree~$3$ ($v_0$ and $u_0$) that correspond to three pairs of adjacent edges each. Hence there are $n-1$ pairs of adjacent edges, so $e_2(A_{m_1,m_2})=\binom{n-3}{2}-(n-1)$. In the second case, the number of matchings is $|E(C_3)||E(A_{m_1,m_2})=3(n-3)$. Therefore, $n_G(2P_2)=\binom{n-3}{2}-(n-1)+3(n-3)$.
    \item $P_3\cup K_1$: We count the number of $P_3$ subgraphs according to the central vertex (the vertex of degree~$2$) in $P_3$. Each vertex of degree~$2$ in $G$ can be the central vertex in exactly one $P_3$ subgraph. There are altogether $n-4$ such vertices. There are also two vertices of degree~$3$ ($v_0$ and $u_0$) that can be the central vertex of three $P_3$ subgraphs each. Hence there are altogether $n+2$ $P_3$ subgraphs. For each $P_3$ subgraph, there are $n-3$ vertices that are not in the $P_3$, and hence can be the $K_1$. Therefore, $n_G(P_3\cup K_1)=(n+2)(n-3)$.
    \item $C_3\cup K_1$: There are two $C_3$ subgraphs in $G$, and for each of them, there are $n-3$ vertices that are not in the $C_3$, and hence can be the $K_1$. Therefore, $n_G(C_3\cup K_1)=2(n-3)$.
    \item $P_4$: This can only be a subgraph of $A_{m_1,m_2}$. We count the number of $P_4$ subgraphs according to the central edge (the edge incident to two vertices of degree~$2$) in $P_4$. The edges $vv_0$ and $vu_0$ can each be the central edge in exactly one $P_4$ subgraph. The edge $v_0u_0$ can be the central edge in three $P_4$ subgraphs.
    
    If $m_1\geq 2$ then $v_0v_1$ can be the central edge in two $P_4$ subgraphs, and $v_{i-1}v_i$ can be the central edge in one $P_4$ subgraph for $2\leq i\leq m_1-1$. However, if $m_1 = 1$ then we do not have any of these subgraphs. Then number of $P_4$ subgraphs here is thus $m_1-1+\alpha_1$ where\[   
    \alpha_1 = 
     \begin{cases}
       1 &\quad\text{if }m_1\geq 2 \\
       0 &\quad\text{if }m_1=1. \\
     \end{cases}
    \]
    A similar argument holds for the edges $u_{i-1}u_i$ for $1\leq i\leq m_2-1$. Therefore, $n_G(P_4)=5+(m_1-1+\alpha_1)+(m_2-1+\alpha_2)=n-3+\alpha_1+\alpha_2$ since $m_1+m_2=n-6$.
    \item $K_{1,3}$: There are only two vertices ($v_0$ and $u_0$) that can be the vertex of degree~$3$ in $K_{1,3}$, and each of them correspond to exactly one $K_{1,3}$ subgraph.
    \item $D_4$  There are only two vertices ($v_0$ and $u_0$) that can be the vertex of degree~$3$ in $D_4$, and each of them correspond to exactly one $D_4$ subgraph.
    \item $C_4$: There are none.
\end{itemize}

The results are summarised in the following table.
\begin{center}
\begin{tabular}{c|c}
    Subgraph $H$ & $n_G(H)$ \\
    \hline\hline
    $2P_2$ & $\begin{aligned} & e_2(A_{m_1,m_2})+3|E(A_{m_1,m_2})| \\ = & \displaystyle\binom{n-3}{2}-(n-1)+3(n-3)\end{aligned}$ \\
    \hline
    $P_3\cup K_1$ & $(n+2)(n-3)$ \\
    \hline
    $C_3\cup K_1$ & $2(n-3)$ \\
    \hline
    $P_4$ & $n-3+\alpha_1+\alpha_2$ \\
    \hline
    $K_{1,3}$ & $2$ \\
    \hline
    $D_4$ & $2$ \\
    \hline
    $C_4$ & $0$
\end{tabular}
\end{center}
where
\[   
\alpha_i = 
     \begin{cases}
       1 &\quad\text{if }m_i\geq 2 \\
       0 &\quad\text{if }m_i=1 \\
     \end{cases}
\]
for $i\in\{1,2\}$. 

Lemma~\ref{lem:i4} gives us $\alpha_1+\alpha_2=1$ so, without loss of generality, $m_1\geq 2$ is even and $m_2=1$.

Applying Proposition~\ref{prop:delvert} to vertex $u_1$, we have
$$I(A_{m_1,1},x)=I(D_{m_1+3},x)+xI(P_{m_1+2},x)=I(C_{m_1+3},x)+xI(P_{m_1+2},x),$$ so the leading coefficient of $I(A_{m_1,1},x)$ is $\frac{1}{2}m_1+2$. Therefore, the leading coefficient of
$$I(G,x)=I(C_3,x)I(A_{m_1,1},x)=(1+3x)I(A_{m_1,1},x)$$
is $\frac{3}{2}m_1+6$. Equating this to the leading coefficient of $I(C_n,x)$, which is $n$, and recalling that $n=m_1+7$ gives us $m_1=2$ or $n=9$ as the only solution. Explicit computation shows that $C_9$ is independence equivalent to $C_3 \cup A_{2,1}$. As we have seen, $A_{2,1}$ is itself independence equivalent to $E_{1,2}$ and $E_{2,1}$. These are the graphs $G_a$, $G_b$ and $G_c$ listed in Theorem~\ref{mytheorem}, and are the cases found by Oboudi~\cite{oboudi}.

\subsection{The subcase $G=C_3 \cup B_{m_1,m_2,m_3}$}\label{sec:notmult3-2}

We tabulate the number of each subgraph with~2 to~4 vertices for $G=C_3 \cup B_{m_1,m_2,m_3}$. (Here, $n=|V(G)|=|V(B_{m_1,m_2,m_3})|+3$.)

\begin{center}
\begin{tabular}{c|c}
    Subgraph $H$ & $n_G(H)$ \\
    \hline\hline
    $2P_2$ & $\begin{aligned}&e_2(B_{m_1,m_2,m_3})+3|E(B_{m_1,m_2,m_3})| \\ = & \displaystyle\binom{n-3}{2}-(n-1)+3(n-3)\end{aligned}$ \\
    \hline
    $P_3\cup K_1$ & $(n+2)(n-3)$ \\
    \hline
    $C_3\cup K_1$ & $2(n-3)$ \\
    \hline
    $P_4$ & $n-4+\alpha_1+\alpha_2+\alpha_3$ \\
    \hline
    $K_{1,3}$ & $2$ \\
    \hline
    $D_4$ & $1$ \\
    \hline
    $C_4$ & $0$
\end{tabular}
\end{center}
where
\[   
\alpha_1 = 
     \begin{cases}
       1 &\quad\text{if }m_1=0 \\
       0 &\quad\text{if }m_1\geq1 \\
     \end{cases}
\]
and
\[   
\alpha_i = 
     \begin{cases}
       1 &\quad\text{if }m_i\geq 2 \\
       0 &\quad\text{if }m_i=1 \\
     \end{cases}
\]
for $i\in\{2,3\}$. 

Lemma~\ref{lem:i4} gives us $\alpha_1+\alpha_2+\alpha_3=1$ so either $m_1=0$ and $m_2=m_3=1$ or, without loss of generality, $m_1\geq 1$, $m_2\geq 2$ and $m_3=1$. The former case gives us $G=C_3\cup B_{0,1,1}$, and explicit computation of the independence polynomial shows that it is indeed independence equivalent to $C_9$. This is the graph $G_d$ mentioned in Theorem~\ref{mytheorem}.

In the latter case, since $m_3=1$ is odd, $m_1$ and $m_2$ must have different parity. Now 
\begin{eqnarray*}
I(B_{m_1,m_2,1},x) & = & I(D_{m_1+m_2+4},x)+xI(D_{m_1+3},x)I(P_{m_2},x)\\
               & = & I(C_{m_1+m_2+4},x)+xI(C_{m_1+3},x)I(P_{m_2},x)
\end{eqnarray*}
Suppose $m_1$ is odd and $m_2$ is even. Then the leading coefficient of $I(B_{m_1,m_2,1},x)$ is $2(\frac{1}{2}m_2+1)=m_2+2$, which is even. Therefore, the leading coefficient of
$$I(G,x) = I(C_3,x)I(B_{m_1,m_2,1},x)$$
is also even (in fact, it is $3m_2+6$). However, the leading coefficient of $I(C_n,x)$ is $n$, an odd number. Therefore, this cannot be the case.

Suppose $m_1$ is even and $m_2$ is odd. Then the leading coefficient of $I(B_{m_1,m_2,1},x)$ is $m_1+3$ so the leading coefficient of $I(G,x)$ is $3(m_1+3)$. But this has to be equal to the leading coefficient of $I(C_n,x)$, which is $n=m_1+m_2+8$. Therefore, $m_2 = 2m_1+1$ and $n=3m_1+9$.

We have
\begin{eqnarray*}
I(B_{m_1,m_2,1},x) & = & I(C_{m_1+m_2+4},x)+xI(C_{m_1+3},x)I(P_{m_2},x)\\
               & = & \left(1+\ldots+(m_1+m_2+4)x^\frac{m_1+m_2+3}{2}\right)\\
                 & & +x\left(1+\ldots+\frac{(m_1+3)^3-(m_1+3)}{24}x^\frac{m_1}{2}+(m_1+3)x^\frac{m_1+2}{2}\right)\\
                 & & \left(1+\ldots+\frac{(m_2+3)(m_2+1)}{8}x^\frac{m_2-1}{2}+x^\frac{m_2+1}{2}\right)
\end{eqnarray*}
and therefore the coefficient of $x^{\frac{1}{2}(m_1+m_2+3)}$ in $I(B_{m_1,m_2,1},x)$ is
\begin{eqnarray*}
& & m_1+m_2+4+\frac{(m_1+3)^3-(m_1+3)}{24}+\frac{(m_1+3)(m_2+3)(m_2+1)}{8}\\
& = & \frac{13m_1^3+93m_1^2+287m_1+279}{24}.
\end{eqnarray*}
The coefficient of $x^{\frac{1}{2}(m_1+m_2+5)}=x^\frac{1}{2}(n-3)$ in $I(G,x) = I(C_3,x)I(B_{m_1,m_2,1},x)$ is thus
$$3\cdot\frac{13m_1^3+81m_1^2+230m_1+216}{24}+1\cdot(m_1+3).$$
This has to be equal to the coefficient of $x^\frac{1}{2}(n-3)$ in $I(C_n,x)$ which is
$$\frac{n^3-n}{24}=\frac{(3m_1+9)^3-(3m_1+9)}{24}$$.
This ultimately simplifies to
$$m_1^3 - m_1 = 0.$$
Since $m_1$ is an even integer, the only possibility is $m_1=0$, contradicting the supposition that $m_1 \geq 1$.

\subsection{The subcase $G=C_3 \cup C_m \cup A_{m_1,m_2}$}\label{sec:notmult3-3}

We tabulate the number of each subgraph with~2 to~4 vertices for $G=C_3 \cup C_m \cup A_{m_1,m_2}$. Since $m$ is an odd number not divisible by $3$, $m\geq 5$. Also, $m_1$ and $m_2$ are both odd. (Here, $n=|V(G)|=|V(A_{m_1,m_2})|+m+3$.)

\begin{center}
\begin{tabular}{c|c}
    Subgraph $H$ & $n_G(H)$ \\
    \hline\hline
    $2P_2$ & $\begin{aligned}  & e_2(C_m)+e_2(A_{m_1,m_2})\\&+3m+(3+m)|E(A_{m_1,m_2})| \\ = & \displaystyle\frac{m(m-3)}{2}+\displaystyle\binom{n-3-m}{2}-(n-1-m)\\&+3m+(3+m)(n-3-m)\end{aligned}$ \\
    \hline
    $P_3\cup K_1$ & $(n+2)(n-3)$ \\
    \hline
    $C_3\cup K_1$ & $2(n-3)$ \\
    \hline
    $P_4$ & $n-3+\alpha_1+\alpha_2$ \\
    \hline
    $K_{1,3}$ & $2$ \\
    \hline
    $D_4$ & $2$ \\
    \hline
    $C_4$ & $0$
\end{tabular}
\end{center}
where
\[   
\alpha_i = 
     \begin{cases}
       1 &\quad\text{if }m_i\geq 3 \\
       0 &\quad\text{if }m_i=1 \\
     \end{cases}
\]
for $i\in\{1,2\}$.

Lemma~\ref{lem:i4} gives us $\alpha_1+\alpha_2=1$ so, without loss of generality, $m_1\geq 3$ is odd and $m_2=1$.

Applying Proposition~\ref{prop:delvert} to vertex $u_1$, we have
\begin{eqnarray*}
I(A_{m_1,1},x) & = & I(D_{m_1+3},x)+xI(P_{m_1+2},x)\\
               & = & I(C_{m_1+3},x)+xI(P_{m_1+2},x)\\
               & = & \left(1+\ldots+2x^\frac{m_1+3}{2}\right)+x\left(1+\ldots+  \binom{\frac{m_1+5}{2}}{2}x^\frac{m_1+1}{2}+x^\frac{m_1+3}{2}\right)\\
               & = & 1+\ldots+\left(2+\frac{(m_1+5)(m_1+3)}{8}\right)x^\frac{m_1+3}{2}+x^\frac{m_1+5}{2}
\end{eqnarray*}
Comparing the leading coefficients of $I(C_n,x)$ and 
$$I(G,x) = I(C_3,x)I(C_m,x)I(A_{m_1,1},x),$$
gives us $n=3m$.

Comparing the coefficients of $x^\frac{n-3}{2}$ in $I(C_n,x)$ and $I(G,x)$ gives us
$$m+\frac{m^3-m}{8}+3m\left(2+\frac{(m_1+5)(m_1+3)}{8}\right)=\frac{n^3-n}{24}.$$
Substituting $n=3m$ and $m_1=n-m-7$, this equation reduces to
$$(m-5)(m-4)=0.$$
Since $m$ is odd, $m=5$ is the only solution. This gives us $G=C_3 \cup C_5 \cup A_{3,1}$. Explicit computation shows that $G$ is indeed independence equivalent to $C_{15}$. As we have seen, $A_{3,1}$ is itself independence equivalent to $E_{1,3}$ and $E_{3,1}$. These are the graphs $G^\prime_a$, $G^\prime_b$ and $G^\prime_c$ listed in Theorem~\ref{mytheorem}, and were first discovered by Beaton, Brown and Cameron~\cite{beaton}.

\subsection{The subcase $G=C_3 \cup C_m \cup B_{m_1,m_2,m_3}$}\label{sec:notmult3-4}

We tabulate the number of each subgraph with~2 to~4 vertices and at most~4 edges for $G=C_3 \cup C_m \cup B_{m_1,m_2,m_3}$. (Here, $n=|V(G)|=|V(B_{m_1,m_2,m_3})|+m+3$.) 

\begin{center}
\begin{tabular}{c|c}
    Subgraph $H$ & $n_G(H)$ \\
    \hline\hline
    $2P_2$ & $\begin{aligned}  & e_2(C_m)+e_2(B_{m_1,m_2,m_3})\\&+3m+(3+m)|E(B_{m_1,m_2,m_3})| \\ = & \displaystyle\frac{m(m-3)}{2}+\displaystyle\binom{n-3-m}{2}-(n-1-m)\\&+3m+(3+m)(n-3-m)\end{aligned}$ \\
    \hline
    $P_3\cup K_1$ & $(n+2)(n-3)$ \\
    \hline
    $C_3\cup K_1$ & $2(n-3)$ \\
    \hline
    $P_4$ & $n-4+\alpha_1+\alpha_2+\alpha_3$ \\
    \hline
    $K_{1,3}$ & $2$ \\
    \hline
    $D_4$ & $1$ \\
    \hline
    $C_4$ & $0$
\end{tabular}
\end{center}
where
\[   
\alpha_1 = 
     \begin{cases}
       1 &\quad\text{if }m_1=0 \\
       0 &\quad\text{if }m_1\geq1 \\
     \end{cases}
\]
and
\[   
\alpha_i = 
     \begin{cases}
       1 &\quad\text{if }m_i\geq 2 \\
       0 &\quad\text{if }m_i=1 \\
     \end{cases}
\]
for $i\in\{2,3\}$. 

Equating $i_4(C_n)$ and $i_4(G)$ gives us $\alpha_1+\alpha_2+\alpha_3=1$ so either $m_1=0$ and $m_2=m_3=1$ or, without loss of generality, $m_1\geq 1$, $m_2\geq 2$ and $m_3=1$. The former case is impossible as the total number of vertices of $G$ is even whereas the total number of vertices of $C_n$ is odd. In the latter case, $m_1$ and $m_2$ are both odd.

Applying Proposition~\ref{prop:delvert} to vertex $w_1$, we have
\begin{eqnarray*}
I(B_{m_1,m_2,1},x) & = & I(D_{m_1+m_2+4},x)+xI(D_{m_1+3},x)I(P_{m_2},x)\\
               & = & I(C_{m_1+m_2+4},x)+xI(C_{m_1+3},x)I(P_{m_2},x)\\
               & = & \left(1+\ldots+2x^\frac{m_1+m_2+4}{2}\right)\\
               & & +x\left(1+\ldots+2x^\frac{m_1+3}{2}\right)\left(1+\ldots+x^\frac{m_2+1}{2}\right)\\
               & = & 1+\ldots+2x^\frac{m_1+m_2+6}{2}
\end{eqnarray*}
The leading coefficient of $I(B_{m_1,m_2,1},x)$ is $2$, an even number. Therefore, the leading coefficient of
$$I(G,x) = I(C_3,x)I(C_m,x)I(B_{m_1,m_2,1},x)$$
is also even (in fact, it is $6m$). However, the leading coefficient of $I(C_n,x)$ is $n$, an odd number. This contradiction means that this case is impossible.

This concludes the proof of Theorem~\ref{mytheorem}. In Section~\ref{sec:notmult3}, it was shown, for odd $n$ that were not multiples of~$3$, that $\mathcal{I}(C_n)=\{C_n,D_n\}$. In Section~\ref{sec:mult3}, we searched for graphs in $\mathcal{I}(C_n)$ that are not isomorphic to $C_n$ or $D_n$, leading to the four subcases. Subcases~\ref{sec:notmult3-1} and~\ref{sec:notmult3-2} turned up graphs in $\mathcal{I}(C_9)$, while Subcase~\ref{sec:notmult3-3} turned up graphs in $\mathcal{I}(C_{15})$. Subcase~\ref{sec:notmult3-4} turned out to be impossible.

\printbibliography
\end{document}